\documentclass[reqno,12pt]{amsart}
\usepackage[margin=1in]{geometry}
\usepackage{amsfonts}
\usepackage[all]{xy}
\usepackage{enumitem}
\usepackage{mathtools}
\usepackage{amssymb,amsmath,url}
\usepackage{setspace}
\usepackage{amsthm}
\usepackage{cite}
\usepackage{color}
\usepackage{mathrsfs}
\numberwithin{equation}{section}
\usepackage[normalem]{ulem}
\usepackage{graphicx}
\usepackage{lmodern}
\graphicspath{ {images/} }
\usepackage{pdfpages}
\usepackage{MnSymbol}
\newcommand{\nm}{\,\rule[-.6ex]{.13em}{2.3ex}\,}
\newcommand{\bignm}{\,\rule[-1.2ex]{.13em}{3.6ex}\,}
\author{Zhenghui Huo and Brett D. Wick}
\title[Weighted estimates of the Bergman projection with matrix weights]{Weighted estimates of the Bergman projection with matrix weights}
\begin{document}
	\thanks{BDW's research is partially supported by National Science Foundation grants DMS \# 1800057 and Australian Research Council DP 190100970.}
\address{Zhenghui Huo, Department of Mathematics and Statistics, The University of Toledo,  Toledo, OH 43606-3390, USA}
\email{zhenghui.huo@utoledo.edu}
\address{Brett D. Wick, Department of Mathematics and Statistics, Washington University in St. Louis,  St. Louis, MO 63130-4899, USA}
\email{bwick@wustl.edu}
		\newtheorem{thm}{Theorem}[section]
	\newtheorem{cl}[thm]{Claim}
	\newtheorem{lem}[thm]{Lemma}
		\newtheorem{ex}[thm]{Example}
	\newtheorem{de}[thm]{Definition}
		\newtheorem{prop}[thm]{Proposition}
	\newtheorem{co}[thm]{Corollary}
	\newtheorem*{thm*}{Theorem}
	\theoremstyle{definition}
		\newtheorem{rmk}[thm]{Remark}
		
	\maketitle
\begin{abstract}
We establish a weighted inequality for the Bergman projection with matrix weights for a class of pseudoconvex domains. We extend a result of Aleman-Constantin  and obtain the following estimate for the weighted norm of $P$:
\[\|P\|_{L^2(\Omega,W)}\leq C(\mathcal B_2(W))^{{2}}.\]
Here $\mathcal B_2(W)$  is the Bekoll\'e-Bonami constant for the matrix weight $W$ and $C$ is a constant that is independent of the weight $W$ but depends upon the dimension and the domain. 
\medskip

\noindent
{\bf AMS Classification Numbers}: 32A25, 32A36,  32A50, 42B20, 42B35

\medskip

\end{abstract}

\section{Introduction}
The purpose of this paper is to investigate weighted inequalities with matrix weights for the Bergman projection on certain class of pseudoconvex domains. 

In harmonic analysis, weighted inequalities characterize weighted spaces on which singular integral operators remain bounded and describe the norm dependence on the weights. For the case of the scalar-valued function with scalar weights, it is well-known that for $1<p<\infty$, a Calderon-Zygmund singular integral operator is bounded over the weighted space $L^p(\mathbb R^N,W)$ with  the weight function  $W$ satisfying the Muckenhoupt $\mathcal A_p$ condition, i.e. the $\mathcal A_p$ constant 
\[\mathcal A_p(W):=\sup_{B}\langle W\rangle^\frac{1}{p}_{B}\langle W^{-\frac{1}{p-1}}\rangle^{\frac{1}{p^\prime}}_{B}<\infty.\]
Here the supremum is taken over all Euclidean balls $B$ in $\mathbb R^N$ and $\langle W\rangle_B$ denotes the average of $W$ over the set $B$ with respect to the Lebesgue measure.  See for example \cite{HMW}.

As to the dependence of  the norm of the operator on $\mathcal A_p(W)$, extensive studies have been made in recent decades. For $p=2$, the $\mathcal A_2$ conjecture speculates that the dependence is linear. The conjecture was solved by Wittwer for the martingale transform  \cite{Wittwer}, by  Petermichl-Volberg for the Beurling operator, by Petermichl for the Hilbert transform in \cite{Petermichl1}, and then the general Calderon-Zygmund  operators in \cite{Hytonen}. More elementary proofs were also obtained by Lerner \cite{Lerner13} and Lacey \cite{Lacey2017}.

One direction to extend weighted theory  considers the setting of weighted $L^p$ space of vector-valued functions with matrix-valued weight $W$.
Let $\Omega$ be a domain. Let $W$ be a locally integrable function on $\Omega$ with its range in the set of positive-semidefinite $d\times d$ matrices. The weighted space $L^2(\Omega, W)$ is the space of  vector-valued measurable functions $f:\Omega\to \mathbb C^d$ for which 
\[\|f\|^2_{L^2(\Omega,W)}:=\int_{\Omega}\langle W(z) f(z),f(z)\rangle dV(z)<\infty;\]
where $\langle\cdot,\cdot\rangle$ is the usual inner product of $\mathbb C^d$.

For the vector-valued function setting with matrix weights, Treil and Volberg \cite{TV1997,TV97} first gave a matrix $\mathcal A_2$ condition and showed it is a necessary and sufficient for the boundedness of the Hilbert transform \cite{TV1997,TV97}. Then Nazarov-Treil \cite{NT96} and Volberg  \cite{Volberg97} separately  generalized this result  and established the boundedness of classical Calderon-Zygmund operators for  matrix $\mathcal A_p$ weights. Christ-Goldberg \cite{CG01} and Goldberg \cite{Goldberg03} studied a class of weighted, vector analogues of the Hardy-Littlewood maximal function and used them to establish the boundedness of a class of singular integral operators on $L^p(W)$, where $W$ is a matrix $\mathcal A_p$ weight. Recently,  progress has also been made on norm dependence. Bickel-Wick \cite{BW2016} and Isralowitz-Kwon-Pott \cite{IKP} separately showed the following estimates for a sparse operator $S$:
\[\|S\|_{L^2(W)}\lesssim \mathcal A_2(W)^{\frac{3}{2}}.\]
Bickel-Petermichl-Wick \cite{BPW} obtained weighted $L^2$ estimates for the Hilbert transform with the norm bound $\mathcal A_2(W)^{\frac{3}{2}}\log(1+ \mathcal A_2(W))$.  Nazarov-Petermichl-Treil-Volberg \cite{NPTV} then improved their result and obtained a better bound $\mathcal A_2(W)^{\frac{3}{2}}$ for Calderon-Zygmund operators as well as Haar shifts and paraproducts. The sharp bound was established to some particular cases such as sparse operators with a simple sparse family by Nazarov-Petermichl-Treil-Volberg \cite{NPTV} and the dyadic square function by Hyt\"onen-Petermichl-Volberg  \cite{HPV}. It is worth noting that a crucial ingredient in some of these results is the reverse H\"older inequality for a scalar $\mathcal A_2$ weight.

Let $\Omega\subseteq \mathbb C^n$ be a bounded domain. Let $dV$ denote the Lebesgue measure. The Bergman projection $P$ is the orthogonal projection from $L^2(\Omega)$ onto the Bergman space $A^2(\Omega)$, the space of all square-integrable holomorphic functions. Associated with $P$, there is a unique function $K_\Omega$ on $\Omega\times\Omega$ such that for any $f\in L^2(\Omega)$:
\begin{equation}
P(f)(z)=\int_{\Omega}K_\Omega(z;\bar w)f(w)dV(w).
\end{equation}
Let $P^+$ denote the positive Bergman operator defined by:
\begin{equation}
P^+(f)(z):=\int_{\Omega}|K_\Omega(z;\bar w)|f(w)dV(w).
\end{equation}
In \cite{BB78,Bekolle}, Bekoll\'e and Bonami established the analogue of Muckenhoupt $\mathcal A_p$ condition in the Bergman setting:
\begin{thm}[Bekoll\'e-Bonami]
	Let $T_z$ denote the Carleson tent over $z$ in $\mathbb B_n$ defined as below:
	\begin{itemize}
		\item $T_z:=\left\{w\in \mathbb B_n:\left|1-\bar w\frac{z}{|z|}\right|<1-|z|\right\}$ for $z\neq 0$, and
		\item $T_z:= \mathbb B_n$ for $z=0$.
	\end{itemize} 	Let the weight $W$ be a positive, locally integrable function on $\mathbb B_n$. Let $1<p<\infty$. Then the following conditions are equivalent:
	\begin{enumerate}[label=\textnormal{(\arabic*)}]
		\item $P:L^p(\mathbb B_n,W)\to L^p(\mathbb B_n,W)$ is bounded;
		\item $P^+:L^p(\mathbb B _n,W)\to L^p(\mathbb B_n,W)$ is bounded;
		\item The Bekoll\'e-Bonami constant $\mathcal B_p(W)$ is finite where: $$\mathcal B_p(W):=\sup_{z\in \mathbb B_n}\langle W\rangle_{T_z}^{\frac{1}{p}}\langle W^{-\frac{1}{p-1}} \rangle^{\frac{1}{p^\prime}}_{T_z}.$$
	\end{enumerate} 
\end{thm}

People have made progress on the dependence of the operator norm $\|P\|_{L^p(\mathbb B_n,W)}$ on $\mathcal B_{p}(W)$. In \cite{Pott}, Pott and Reguera gave a weighted $L^p$ estimate for the Bergman projection on the upper half plane. Their estimates are in terms of the Bekoll\'e-Bonami constant and the upper bound is sharp. Later, Rahm, Tchoundja, and Wick \cite{Rahm} generalized the results of Pott and Reguera to the unit ball case, and also obtained estimates for the Berezin transform. Sharp weighted norm estimates of the Bergman projection have been obtained \cite{ZhenghuiWick2} on the Hartogs triangle and a broad class of pseudoconvex domains \cite{HWW,HWW2,Hu}.

 Unlike Muckenhoupt $\mathcal A_2$ weights, the reverse H\"older inequality is generally not available  for B\'ekoll\'e-Bonami weights, thereby making weighted inequalities for the Bergman projection harder to prove in the matrix weight setting. Nevertheless, Aleman and Constantin \cite{AC} extended Bekoll\'e and Bonami's result on the unit disc to operator-valued weights for $p=2$. They showed that a $\mathcal B_2$ condition for operator weights determines the boundedness of the Bergman projection in vector-valued $L^2$-spaces with operator-valued weights as opposed to just matrix valued weights (i.e. they consider the action on an infinite dimensional Hilbert space as opposed to just $\mathbb C^d$). They  also obtained a weighted norm estimate for the projection:
\[\|P\|_{L^2(\mathbb D, W)}\lesssim \mathcal B^{5/2}_2(W).\]
In their approach, they related the norms of analytic functions in weighted Bergman spaces to weighted norms of their derivatives, which relied on various properties of  the unit disc $\mathbb D$. 

In this paper, we consider the weighted estimates of the Bergman projection in the matrix weight case.  We follow the settings from \cite{HWW,HWW2} and the domain $\Omega$ we consider  belong to one of the following classes:
\begin{itemize}
	\item a bounded, smooth, pseudoconvex domain of finite type in $\mathbb C^2$, \item a bounded, smooth, strictly pseudoconvex domain in $\mathbb C^n$, 
	\item a bounded, smooth, convex domain of finite type in $\mathbb C^n$, or
	\item a bounded, smooth, decoupled domain of finite type in $\mathbb C^n$.
\end{itemize}
Such a $\Omega$ is refered to as a ``simple domain''  by McNeal in \cite{McNeal2003}. We will use the same terminology in this paper. 

Let $\Omega$ be a simple domain.
The main result can be summarized as follows:
\begin{thm}\label{thm1.2}
	 Let $W$ be a matrix $\mathcal B_2$ weight on $\Omega$. Let $\mathcal B_2(W)$ denote the matrix $\mathcal B_2$ constant  given by
	 \[\mathcal B_2(W):=\sup_{B\in \mathscr{B}}\bignm\langle W\rangle^{1/2}_{B}\langle W^{-1}\rangle^{1/2}_{B}\bignm^2,\]
	 where $\nm\cdot\nm$ denotes the norm of the matrix acting on $\mathbb C^d$, $\langle W\rangle_B:=\frac{\int_BWdV}{\int_BdV}$, and $\mathscr B$ is a collection of tents in $\Omega$ that touch its boundary. Then we have \[\|P\|_{L^2(\Omega,W)}\leq C(\mathcal B_2(W))^{{2}}.\]
The constant  $C$ in the inequality only depends on the domain $\Omega$ and the dimension $d$, but not the weight  $W$.
\end{thm}
This theorem extends and improves the estimates by Aleman-Constantin \cite{AC} by lowering the constant power of $\mathcal B_2(W)$ in the upper bound and showing it works for more domains.
See Section 2.2 for the detailed definition of the tents of $\Omega$ and the collection $\mathscr{B}$. 

Our approach is motivated by ideas from \cite{AC,NPTV,Aleman} and can be outlined as follows:
Given  a matrix $\mathcal B_2$ weight $W$, we construct finitely many  step averaging weights $\mathcal W_l$ {according to dyadic systems $\mathcal T_l$} and their sum $\tilde{\mathcal W}=\sum\mathcal W_l$ with the following three properties: 
\begin{enumerate}
	\item $\sup_{\hat K^k_j\in \mathcal T_l}\nm\langle \mathcal W_l^{-1}\rangle^{1/2}_{\hat K^k_j}\langle \mathcal W_l\rangle^{1/2}_{\hat K^k_j}\nm^2\lesssim \mathcal B_2(W)$ (Lemma \ref{lem4.2}),
	\item for $v\in \mathbb C^d$ and each $\mathcal W_l$, the scalar weight $\langle \mathcal W_lv,v\rangle$ satisfies the reverse H\"older inequality (Lemma \ref{lem4.4}),
	\item \label{(3)}$\|P\|_{L^2(W)}\lesssim \mathcal B^{1/2}_2(W)\|P\|_{L^2(\tilde{\mathcal W})}$ (Theorem \ref{thm4.3}).
	\end{enumerate}
The construction of the step averaging  weight $\mathcal W_l$ is inspired by the work of \cite{Aleman}. Property (\ref{(3)}) above uses a duality argument from \cite{AC} and  relates the weighted norm of $P$ on $L^2(W)$ to the ones on $L^2(\mathcal W_l)$.  Using the known estimates about the Bergman kernel on simple domains, we show that the Bergman projection belongs to a convex body-valued dyadic operator \cite{NPTV}. Since the reverse H\"older is available for $\langle \mathcal W_lv,v\rangle$, we are able to establish weighted estimates using some scalar dyadic square operators. 

After the first draft of this paper was posted, we received a draft by Limani and Pott \cite{LP21} where they considered the Bergman projection on the upper half plane and obtained weighted estimates of the projection with operator-valued weights. Their results sharpen the exponent of 2 in the estimate of Theorem \ref{thm1.2} to $\frac{3}{2}$.

Our paper is organized as follows: In Section 2, we recall the definitions and known results about  the dyadic tent structure of $\Omega$,  estimates for the Bergman kernel function.  In Section 3, we give the definition of the convex body-valued sparse operator and show the convex body domination for the Bergman projection.  In Section 4, we introduce the weight $\mathcal W$ and go over its properties. In Section 5, we prove Theorem \ref{thm1.2}. We make several remarks for our results in Section 6.

Throughout the paper, $n$ will be the dimension of the complex Euclidean space $\mathbb C^n$ that contains the domain $\Omega$ of vector-valued functions and $d$ will be the dimension of $\mathbb C^d$ in which vector-valued functions take their range. 
Hence the matrix weights will have range in the set of $d\times d$ matrices. We use the notation $\|\cdot\|$ for the norm of operators/functions on function spaces, and use $\nm\cdot\nm$ for the  norm of the matrix acting on $\mathbb C^d$ or the length of a vector in $\mathbb C^d$. Given functions of several variables $f$ and $g$, we use $f\lesssim g$ to denote that $f\leq Cg$ for a constant $C$. If $f\lesssim g$ and $g\lesssim f$, then we say $f$ is comparable to $g$ and write $f\approx g$. 
\section{Preliminaries}

\subsection{Balls on the boundary of a simple domain} Let $\Omega$ be a bounded domain in $\mathbb C^n$ with smooth boundary $\mathbf b\Omega$. Then there is a smooth function $\rho$ satisfying  $\Omega=\{z\in \mathbb C^n:\rho(z)<0\}$ and  $\rho\neq 0$ on $\mathbf b\Omega$. We call such a function $\rho$ a \emph{defining function} of $\Omega$.
$\Omega$ is a \emph{pseudoconvex domain} if for any $p\in \mathbf b\Omega$, the complex Hessian $(\rho_{i\bar j})$ of the defining function is positive semi-definite on holomorphic tangent space  $T^{(1,0)}_p(\mathbf b\Omega)$ at $p$. For $m>0$, a point $p\in \mathbf b\Omega$ is of \emph{finite type}  in the sense of D'Angelo \cite{D'Angelo82} if the maximum order of contact of one-dimensional complex analytic varieties with $\mathbf \Omega$ at $p$ equals $m$. We say a domain $\Omega$ is of  \emph{finite type} if every boundary point $p$ is  of finite type.
 
When $\Omega$ is a simple domain  in $\mathbb C^n$, non-isotropic sets can be constructed using a special  coordinate system near the boundary of $\Omega$.  See results of McNeal \cite{McNeal2,McNeal91,McNeal2003}. Let $p\in \mathbf b\Omega$ be a point of finite type $m$. For a small neighborhood $U$ of the point $p$, there exists a holomorphic coordinate system $z=(z_1,\dots,z_n)$ centered at a point $q_0\in U$ and defined on $U$ and quantities $\tau_1(q,\delta), \tau_2(q,\delta),\dots, \tau_n(q,\delta)$ for each $q\in U$ such that 
\begin{align}\label{2.10}\;\;\;\;\;\;\tau_1(q,\delta)=\delta\;\;\; \text{ and }\;\;\;\delta^{1/2}\lesssim\tau_j(q,\delta)\lesssim\delta^{1/m}\;\;\text{ for }\;\; j=2,3,\dots,n.\end{align}
Moreover, the polydisc $D(q,\delta)$ defined by:
\begin{align}
D(q,\delta)=\{z\in\mathbb C^n:|z_j|<\tau_j(q,\delta),j=1,\dots,n\}
\end{align}
is the largest one centered at $q$ on which the defining function $\rho$ changes by no more than $\delta$ from its value at $q$, i.e. if $z\in D(q,\delta)$, then $|\rho(z)-\rho(q)|\lesssim \delta$.

The polydisc $D(q,\delta)$ is known to satisfy several ``covering properties'' \cite{McNeal3}: \begin{enumerate}\item There exists a constant $C>0$, such that for points $q_1,q_2\in U\cap \Omega$ with $D(q_1,\delta)\cap D(q_2,\delta)\neq \emptyset$, we have
	\begin{align}\label{2.11}
	D(q_2,\delta)\subseteq CD(q_1,\delta) \text{ and }  D(q_1,\delta)\subseteq CD(q_2,\delta).
	\end{align}
	\item There exists a constant $c>0$ such that for $q\in U\cap \Omega$ and $\delta>0$, we have \begin{align}\label{2.12}D(q,2\delta)\subseteq cD(q,\delta).\end{align}\end{enumerate}
It was also shown in \cite{McNeal3} that $D(p,\delta)$ induces a global quasi-metric on $\Omega$. Here we will use it to define a quasi-metric on $\mathbf b\Omega$.

For $q\in \mathbf b\Omega$ and $\delta>0$, we define the non-isotropic ball of radius $\delta$ to be the set
\[ B(q,\delta)={D(q,\delta)}\cap\mathbf b\Omega.\]
Set containments (\ref{2.11}), (\ref{2.12}), and the compactness and smoothness of $\mathbf b\Omega$  imply the following properties for the balls:
\begin{enumerate}\item There exists a constant $C$ such that for $q_1,q_2\in U\cap \mathbf b\Omega$ with $B(q_1,\delta)\cap B(q_2,\delta)\neq \emptyset$, \begin{align}\label{2.14}
	B(q_2,\delta)\subseteq CB(q_1,\delta) \text{ and }  B(q_1,\delta)\subseteq CB(q_2,\delta).
	\end{align} \item Let $\mu$ denote the Lebesgue surface area on $\mathbf b\Omega$. There exists a constant $c>0$ such that for $q\in U\cap \mathbf b\Omega$ and $\delta>0$, we have \begin{align}\label{2.15}B(q,\delta)\subseteq cB(q,\delta/2)\;\;\;\;\;\text{ and }\;\;\;\;\;\mu(B(q,\delta))\approx \prod_{j=2}^{n}\tau_j^2(q,\delta).\end{align} \end{enumerate}

The balls $B$ induce a quasi-metric on $\mathbf b\Omega\cap U$. For $q,p\in \mathbf b \Omega\cap U$, we set
$\tilde d(q,p)=\inf\{\delta>0:p\in B(q,\delta)\}.$ Note that $\mathbf b\Omega$ is coompact. To extend this quasi-metric $\tilde d(\cdot,\cdot)$ to a global quasi-metric $d(\cdot,\cdot)$ defined on $\mathbf b\Omega\times\mathbf b\Omega$, one can just patch the local metrics together in an appropriate way. The resulting quasi-metric is not continuous, but satisfies all the relevant properties. We refer the reader to \cite{McNeal3} for more details on this matter. Since $d(\cdot,\cdot)$ and $\tilde d(\cdot,\cdot)$ are equivalent, we may abuse the notation $B$ for the ball on the boundary induced by $d$. Then (\ref{2.14}) and (\ref{2.15}) still hold true for $B$.

\subsection{Dyadic tents on $\Omega$ and the matrix $\mathcal B_2$ constant} The non-isotropic ball $B(p,\delta)$ on the boundary $\mathbf b\Omega$ induces ``tents'' in the domain $\Omega$.
To define what ``tents'' are we need the orthogonal projection map near the boundary. Let $\operatorname{dist}(\cdot,\cdot)$ denote the Euclidean distance in $\mathbb C^n$. For small $\epsilon>0$, set \begin{align*}&N_{\epsilon}(\mathbf b\Omega)=\{w\in \mathbb C^n: \operatorname{dist}(w,\mathbf b\Omega)<\epsilon\}.\end{align*}
\begin{lem}\label{lem3.2}
	For sufficiently small $\epsilon_0>0$, there exists a map $\pi:N_{\epsilon_0}(\mathbf b\Omega)\to \mathbf b\Omega$  such that
	\begin{enumerate}[label=\textnormal{(\arabic*)}]
		\item For each point $z\in N_{\epsilon_0}(\mathbf b\Omega)$ there exists a unique point $\pi(z)\in \mathbf b\Omega$ such that \[\operatorname{dist}(z,\pi(z))=\operatorname{dist}(z,\mathbf b\Omega).\]
		\item For $p\in \mathbf b\Omega$, the fiber $\pi^{-1}(p)=\{p-\epsilon n(p): -\epsilon_0< \epsilon<\epsilon_0\}$ where $n(p)$ is the outer unit normal vector of $\mathbf b\Omega$ at point $p$.
		\item The map $\pi$ is smooth on $N_{\epsilon_0}(\mathbf b\Omega)$.
		\item If the defining function $\rho$ is the signed distance function to the boundary, the gradient $\nabla\rho$ satisfies \[\nabla\rho(z)=n(\pi(z)) \;\text{ for } \;z\in N_{\epsilon_0}(\mathbf b\Omega).\]
	\end{enumerate}
\end{lem}
A proof of Lemma \ref{lem3.2} can be found in \cite{BaloghBonk}.
\begin{de}\label{de3.3}
	Let $\epsilon_0$ and $\pi$ be as in Lemma \ref{lem3.2}. For $z\in \mathbf b\Omega$ and sufficiently small  $\delta>0$, the ``tent'' $B^\#(z,\delta)$ over the ball $B(z,\delta)$ is defined to be the subset of $N_{\epsilon_0}(\mathbf b\Omega)$ as follows:
	When $\Omega$ is a simple domain in $\mathbb C^n$,
	\[B^\#(z,\delta):=\{w\in  \Omega: \pi(w)\in B(z,\delta), \operatorname{dist}(\pi(w),w)< \delta\}.\] 
	For $\delta\gtrsim 1$ and any  $z\in \mathbf b\Omega$, we set $B^\#(z,\delta)=\Omega$.
\end{de}
For the ``tent'' $B^\#(z,\delta)$ to be within $N_{\epsilon_0}(\mathbf b\Omega)$, the constant  $\delta$ in Definition \ref{de3.3} needs to satisfy  $\delta<\epsilon_0$.

Given a subset $U\in \mathbb C^n$, let $|U|$ denote the Lebesgue measure of $U$. By the definitions of the tents $B^\#(z,\delta)$ we have:
\begin{align}\label{3.1}
|B^\#(z,\delta)|\approx \delta^{2}\prod_{j=2}^{n}\tau^2_j(z,\delta)
\end{align}
and hence also the ``doubling property'':
\begin{align}\label{3.3}
|B^\#(z,\delta)|\approx |B^\#(z,\delta/2)|.
\end{align}

Now we are in the position of constructing dyadic systems on $\mathbf b\Omega$ and $\Omega$. 
Note that the ball $B(\cdot,\delta)$ on $\mathbf b\Omega$ satisfies the ``doubling property'' as in (\ref{2.15}). By (\ref{2.14}),  the surface area $\mu(B(q_1,\delta))\approx \mu(B(q_2,\delta))$ for any $q_1,q_2\in \mathbf b\Omega$ satisfying $d(q_1,q_2)\leq \delta$. Combining these facts yields that the metric $d(\cdot,\cdot)$ is a doubling metric, i.e. for every $q\in \mathbf b\Omega$ and $\delta>0$, the ball $B(q,\delta)$ can be covered by at most $M$ balls $B(x_i,\delta/2)$. Results of Hyt\"onen and Kairema in \cite{HK} then give the following lemmas:
\begin{lem}\label{lem3.5}Let $\delta$ be a positive constant that is sufficiently small and let $s>1$ be a parameter.
	There exist reference points $\{p_j^{(k)}\}$ on the boundary $\mathbf b\Omega$ and  an associated collection of subsets $\mathcal Q=\{Q_j^{k}\}$ of $\mathbf b\Omega$ with $p_j^{(k)}\in Q_j^{k}$ such that the following properties hold:
	\begin{enumerate}[label=\textnormal{(\arabic*)}]
		\item For each fixed $k$, $\{p_j^{(k)}\}$ is a largest set of points on $\mathbf b\Omega$ satisfying $d_1(p_j^{(k)},p_i^{(k)})> s^{-k}\delta$ for all $i,j$. In other words, if $p\in \mathbf b\Omega$ is a point that is not in $\{p_j^{(k)}\}$, then there exists an index $j_o$ such that $d_1(p,p_{j_o}^{(k)})\leq s^{-k}\delta$.
		\item For each fixed $k$, $\bigcup_j Q^k_j=\mathbf b\Omega$ and $Q^k_j\bigcap Q^k_i=\emptyset$ when $i\neq j$.
		\item For $k< l$ and any $i,j$, either $Q^k_j\supseteq Q^l_i$ or $Q^k_j\bigcap Q^l_i=\emptyset$.
		\item There exist positive constants $c$ and $C$ such that for all $j$ and $k$, \[B(p_j^{(k)},cs^{-k}\delta)\subseteq Q^k_j\subseteq B(p_j^{(k)},Cs^{-k}\delta).\]
		\item Each $Q_j^k$ contains at most $N$ numbers of $Q^{k+1}_i$. Here $N$ does not depend on $k, j$.
	\end{enumerate}
\end{lem}
The points $p^{(k)}_j$ above are dyadic points in $\mathbf b\Omega$.

\begin{lem}\label{lem3.6}
	Let $\delta$ and $\{p^{(k)}_j\}$ be as in Lemma \ref{lem3.5}. There are finitely many collections $\{\mathcal Q_l\}_{l=1}^{N}$ such that the following hold:\begin{enumerate}[label=\textnormal{(\arabic*)}]\item Each collection $\mathcal Q_l$ is associated to some dyadic points $\{z^{(k)}_j\}$ and satisfies all the properties in Lemma \ref{lem3.5}.\item For any $z\in \mathbf b\Omega$ and small $r>0$, there exist $Q_{j_1}^{k_1}\in \mathcal Q_{l_1}$ and $Q_{j_2}^{k_2}\in \mathcal Q_{l_2}$ such that
		\[Q_{j_1}^{k_1}\subseteq B(z,r)\subseteq Q_{j_2}^{k_2}\;\;\;\text{ and }\;\;\;\mu(B(z,r))\approx\mu(Q_{j_1}^{k_1})\approx\mu(Q_{j_2}^{k_2}).\]\end{enumerate}
\end{lem}
Letting the sets $Q_{j}^{k}$ in Lemma 2.3 serve as these bases, we construct dyadic tents in $\Omega$ as follows:
\begin{de}\label{de3.7}
	Let $\delta$, $\{p^{(k)}_j\}$ and $\mathcal Q=\{Q_j^{k}\}$ be as in Lemma \ref{lem3.5}. We define the collection $ {\mathcal T}=\{\hat {K}_j^{k}\}$ of dyadic tents in the domain $\Omega$ as 
	\[\hat {K}_j^{k}:=\{z\in\Omega: \pi(z)\in Q_j^k \text{ and }\operatorname{dist}(\pi(z),z)<s^{-k}\delta\}.\]
\end{de}
\begin{lem}\label{lem3.8}Let $\mathcal T=\{\hat K^k_j\}$ be a collection of dyadic tents in Definition \ref{de3.7} and let $\{\mathcal Q_l\}_{l=1}^{N}$ be collections  in Lemma \ref{lem3.6}.  The following statements hold true:\begin{enumerate}[label=\textnormal{(\arabic*)}]\item For any $\hat {K}_j^{k}$, $\hat {K}_i^{k+1}$ in $\mathcal T$, either $\hat {K}_j^{k}\supseteq\hat {K}_i^{k+1}$ or $\hat {K}_j^{k}\bigcap\hat {K}_i^{k+1}=\emptyset$.\item For any $z\in \mathbf b\Omega$ and small $r>0$, there exist $Q_{j_1}^{k_1}\in \mathcal Q_{l_1}$ and $Q_{j_2}^{k_2}\in \mathcal Q_{l_2}$ such that
		\[\hat K_{j_1}^{k_1}\subseteq B^\#(z,r)\subseteq \hat K_{j_2}^{k_2}\;\;\;\text{ and }\;\;\;|B^\#(z,r)|\approx |\hat K_{j_1}^{k_1}|\approx |\hat K_{j_2}^{k_2}|.\]\end{enumerate}\end{lem}

 Set $\mathscr{B}:=\{B^\#(z,r)\}_{z\in \mathbf b\Omega,r\leq \epsilon_0}\bigcup\{\Omega\}$. Recall the matrix $\mathcal B_2$ constant given by
\[\mathcal B_2(W):=\sup_{B\in \mathscr{B}}\bignm\langle W\rangle_{B}^{1/2}\langle W^{-1}\rangle_{B}^{1/2}\bignm^2,\]
where $\nm\cdot\nm$ denotes the norm of the matrix acting on $\mathbb C^d$ and $\langle W\rangle_B:=\frac{\int_BWdV}{\int_BdV}$. Let $\{\mathcal T_l\}_{l=1}^N$ be collections of  dyadic tents induced by collections $\{\mathcal Q_l\}_{l=1}^N$ in Lemma \ref{lem3.6} respectively. Set $\mathscr{T}:=\{\hat K^k_j:\hat K^k_j\in \mathcal T_l \text{ for  some }l=1,\ldots,N\}\bigcup\{\Omega\}$. Then Lemma \ref{lem3.8} implies that 
\[\mathcal B_2(W)\approx\sup_{\hat K^k_j\in \mathscr{T}}\bignm\langle W\rangle_{\hat K^k_j}^{1/2}\langle W^{-1}\rangle_{\hat K^k_j}^{1/2}\bignm^2.\]
From now on, we will abuse the notation of $\mathcal B_2(W)$  to represent both the supremum in $\mathscr{B}$ and $\mathscr{T}$.
\subsection{Dyadic kubes on $\Omega$} \begin{de}\label{de3.9}For a collection $\mathcal T$ of dyadic tents, we define the center $\alpha_j^{(k)}$ of each tent $\hat {K}_j^{k}$ to be the point satisfying 
	\begin{itemize}
		\item $\pi(\alpha_j^{(k)})=p^{(k)}_j$; and 
		\item $\operatorname{dist} (p^{(k)}_j,\alpha_j^{(k)})=\frac{1}{2}\sup_{\pi(p)=p^{(k)}_j}\operatorname{dist}(p,\mathbf b\Omega)$.
	\end{itemize}
	We set $K^{-1}=\Omega\backslash\left(\bigcup_{j}\hat  {K}_j^{0}\right)$, and for each point $\alpha_j^{(k)}$ or its corresponding tent $\hat K^k_j$, we define the ``kube''  
	${K}_j^{k}:=\hat {K}_j^{k}\backslash\left(\bigcup_{l}\hat  {K}_l^{k+1}\right),$
	where $l$ is any index with $p^{(k+1)}_l\in \hat{K}^{k}_j$. \end{de}
By the definition of dyadic kubes, the following lemma for dyadic kubes holds true:\begin{lem}\label{3.10} Let $\mathcal T=\{\hat K^k_j\}$ be the system of tents induced by  $\mathcal Q$ in Definition \ref{de3.7}. Let $K^k_j$ be the kubes of $\hat K^k_j$. Then \begin{enumerate}[label=\textnormal{(\arabic*)}]\item $\pi(K^k_j)=Q^k_j\in \mathcal Q$;\item $K^k_j$'s are pairwise disjoint and\; $\bigcup_{j,k}K^k_j=\Omega$; \item the following volume estimates hold: \begin{align}\label{3.6}|K^k_j|\approx |\hat K^k_j|\approx s^{-2k}\delta^{-2}\prod_{j=2}^{n}\tau_j^2(p^{(k)}_j,s^{-k}\delta).\end{align}\end{enumerate}\end{lem}
The next lemma is crucial in the proof of  Lemma \ref{lem4.4} which relates the norm of a holomorphic vector-valued function on $L^2(\tilde{\mathcal  W})$.
\begin{lem}\label{lem2.9}
	There exists finitely many collections of $\{\mathcal T_l\}_{l=1}^M$ such that for any $z\in \Omega$, we can find a tent $\hat K^k_j\in \mathcal T_l$ for some $l$ such that the following holds:
	\begin{enumerate}[label=\textnormal{(\arabic*)}]
		\item $z\in K^k_j$,
		\item If $k\geq 0$,  then $K^k_j\supseteq D(z, cs^{-k})$ for some constant $c>0$. If $k=-1$, then $K^{-1}$  contains a Euclidean ball centered at $z$ with radius $r\approx 1$.
	\end{enumerate}
\end{lem}
\begin{proof}
Let $\{\mathcal T_l\}_{l=1}^N$ be collections of dyadic tents induced by collections $\{\mathcal Q_l\}_{l=1}^N$ in Lemma \ref{lem3.6} respectively with parameter $\delta$ and reference points $p^{(k)}_j$. Now let  $\{\mathcal T^\prime_l\}_{l=1}^N$ and $\{\mathcal T^{\prime\prime}_l\}_{l=1}^N$ be collections also induced by $\{Q_l\}_{l=1}^N$ in Lemma \ref{lem3.6}, but each $\hat K^{\prime k}_j\in \mathcal T^\prime_l$ and $\hat K^{\prime \prime k}_j\in \mathcal T^\prime_l$ are defined by $Q^k_j\in \mathcal Q_l$ as follows:
	\begin{align}&\hat K^{\prime k}_j:=\left\{z\in\Omega: \pi(z)\in Q_j^k \text{ and }\operatorname{dist}(\pi(z),z)<\frac{s+2}{3}\delta s^{-k}\right\},\nonumber\\&\hat K^{\prime\prime k}_j:=\left\{z\in\Omega: \pi(z)\in Q_j^k \text{ and }\operatorname{dist}(\pi(z),z)<\frac{2s+1}{3}\delta s^{-k-1}\right\},\nonumber\end{align}
i.e. $\hat K^{\prime k}_j\in \mathcal T^\prime_l$ and $\hat K^{\prime\prime k}_j\in \mathcal T^{\prime\prime}_l$ have the same ``base'' $Q^k_j$ as the tent $\hat K^{k}_j$ in $\mathcal T_l$ but  different ``height''. Thus $\mathcal T^\prime_l$ and $\mathcal T^{\prime\prime}_l$ still satisfies Lemma \ref{lem3.8}. We call the tents $\hat K^{\prime k}_j$ and $\hat K^{\prime\prime k}_j$  {\it{cousins}} of $\hat K^{k}_j$ and use $K^{\prime k}_j$ and  $K^{\prime\prime k}_j$ to denote the dyadic kubes induced by $\hat K^{\prime k}_j$ and $\hat K^{\prime\prime k}_j$  respectively as in Definition \ref{de3.9}. We claim $\{\mathcal T_l\}_{l=1}^N\bigcup \{\mathcal T^\prime_l\}_{l=1}^N\bigcup \{\mathcal T^{\prime\prime}_l\}_{l=1}^N$ are the desired collections in Lemma \ref{lem2.9}. Hence, $M=3N$.

When the point $z\in \Omega$ is far away from the boundary $\mathbf b \Omega$, say $z\in K^{\prime-1}$. Then $\operatorname{dist}(z,\pi(z))\geq \frac{2s+1}{3}\delta$. Therefore, $\text{dist}(z,\mathbf b K^{-1})\geq \frac{2s-2}{3}\delta$  and $K^{-1}$ contains a ball centered at $z$ with radius $\frac{2s-2}{3}\delta\approx 1$. 

When the point $z\in \Omega\backslash K^{\prime-1}$.  Lemma \ref{lem3.8} implies that there exists a tent $\hat K^k_j\in \mathcal T_l$ with $z\in K^k_j$ and a constant $c_1>0$ such that $\pi(\hat K^k_j)\supseteq B(\pi(z),c_1s^{-k})$. We first show that there exists a constant $c_2$ such that the polydisc $D(z,c_2s^{-k})$ satisfies $\pi(D(z,c_2s^{-k}))\subseteq B(\pi(z),c_1s^{-k})$.  Suppose $\zeta\in D(z,c_2s^{-k})$. Let $\zeta^\prime$ be the projection of $\zeta$ on the hypersurface  $\{w:\rho(w)=\rho(z)\}$. 
Then we have $d(\zeta,z)\lesssim c_2s^{-k}$ and 
\[d(\zeta^\prime,z)\lesssim d(\zeta^\prime,\zeta)+d(z,\zeta)\lesssim c_2s^{-k}.\]
Since the parameter $\delta$ in Lemma \ref{lem3.5} is chosen to be sufficiently small and the gradient $\nabla \rho$ is continuous near $\mathbf b\Omega$, we may assume that $|\nabla\rho(x)-\nabla \rho(y)|<\epsilon$ for all  $y\in \mathbf b\Omega$ and $x\in B(y,\delta)$.
Then, from $\pi(\zeta)=\zeta^\prime -\rho(z)\nabla\rho(\zeta)$ and $\pi(z)=z -\rho(z)(\nabla\rho(z))$ , we obtain
\begin{align*}
d(\pi(\zeta),\pi(z))&\lesssim d(\pi(\zeta),\zeta^\prime-\rho (z)\nabla \rho(z))+d(\pi(z),\zeta^\prime-\rho (z)\nabla \rho(z))\nonumber\\&\lesssim -\epsilon \rho(z)+d(\pi(z),\zeta^\prime-\rho (z)\nabla \rho(z)).
\end{align*} 
By (\ref{2.11}) and (\ref{2.12}), there exists a constant $C>0$ such that $D(z,\delta s^{-k})\subseteq CD(\pi(z), \delta s^{-k})$,
Thus, $D(z,\delta s^{-k})-\rho (z)\nabla \rho(z)\subseteq CD(\pi(z), \delta s^{-k})$. Since the sets $D(z,\delta s^{-k})-\rho (z)\nabla \rho(z)$ and $CD(\pi(z), \delta s^{-k})$ have the common center $\pi(z)$, the containment still holds after applying a dilation to both sets. Therefore, there exists a constant $C_2$ such that
\[C^{-1}C_2^{-1}D(z,\delta s^{-k})-\rho (z)\nabla \rho(z)\subseteq C^{-1}D(\pi(z), \delta s^{-k})\subseteq D(\pi(z),c_1s^{-k}).\]
We may choose $c_2$ to be small so that 
\[D(z,c_2s^{-k})\subseteq C^{-1}C_2^{-1}D(z,\delta s^{-k}).\]
Then we obtain $D(z,c_2s^{-k})-\rho (z)\nabla \rho(z)\subseteq D(\pi(z),c_1s^{-k})$ which also implies that 
\[d(\pi(z),\zeta^\prime-\rho (z)\nabla \rho(z))\lesssim c_1s^{-k}.\]
Hence $d(\pi(\zeta),\pi(z))\lesssim c_1s^{-k}$ and $\pi(D(z,c_2s^{-k}))\subseteq B(\pi(z),c_1s^{-k})$ for some $c_2$.

 Now we are ready to show that $D(z,cs^{-k})$ is in $K^k_j$, $K^{\prime k}_j$, or $K^{\prime\prime k}_j$ for some $c>0$. Since $Q^k_j\supseteq \pi(D(z,c_2s^{-k}))$, it suffices to check the distance from $z$ to the boundary $\mathbf b\Omega$. Note that
 \begin{align*}
& \{-\rho(\zeta):\zeta\in K^k_j\}=\left[\delta s^{-k-1},\delta s^{-k}\right),\nonumber\\
& \{-\rho(\zeta):\zeta\in K^{\prime k}_j\}=\left[ \frac{s+2}{3}\delta s^{-k-1},\frac{s+2}{3}\delta s^{-k}\right),\nonumber\\
& \{-\rho(\zeta):\zeta\in K^{\prime \prime k}_j\}=\left[ \frac{2s+1}{3}\delta s^{-k-2},\frac{2s+1}{3}\delta s^{-k-1}\right).\nonumber
 \end{align*}
If $z\in K^k_j\backslash K^{\prime\prime k}_j$, then the distance from $z$ to the boundary of  $K^{\prime k}_j$ along the normal direction $\pm \nabla \rho(z)$ is greater than $\frac{s-1}{3}\delta s^{-k-1}$. If $z\in K^k_j\backslash K^{\prime k}_j$, the distance from $z$ to the boundary of  $K^{\prime\prime k}_j$ along the normal direction $\pm \nabla \rho(z)$ is greater than $\frac{s-1}{3}\delta s^{-k-2}$. Otherwise, $z\in K^{\prime\prime k}_j\bigcap K^{\prime k}_j$ and the distance from $z$ to the boundary of  $K^{ k}_j$ along the normal direction $\pm \nabla \rho(z)$ is greater than $\frac{s-1}{3}\delta s^{-k-1}$.  Hence by choosing $c=\min\{c_2, \frac{s-1}{3}\delta s^{-2}\}$, we have $D(z,cs^{-k})$ is contained in $K^{ k}_j$, $K^{\prime  k}_j$, or $K^{\prime \prime k}_j$.
\end{proof}
\subsection{Estimates for the Bergman kernel} When $\Omega$ is a simple domain in $\mathbb C^n$, estimates of the Bergman kernel function on $\Omega$ were obtained in \cite{McNeal2,McNeal91,McNeal2003}. 
\begin{thm}\label{thm4.4}
Let $\Omega$ be a simple domain in $\mathbb C^n$. Let $p$ be a boundary point of $\Omega$. There exists a neighborhood $U$ of $p$ so that for all $q_1,q_2\in U\cap \Omega$,
\begin{align}
|K_\Omega(q_1;\bar q_2)|\lesssim t^{-2}\prod_{j=2}^{n}\tau_j(q_1,t)^{-2},
\end{align}
where $t=|\rho(q_1)|+|\rho(q_2)|+\inf\{\epsilon>0:q_2\in D(q_1,\epsilon)\}$.
\end{thm}
We can reformulate Theorem \ref{thm4.4} as below.
\begin{thm}\label{thm4.5}
Let $\Omega$ be a simple  domain in $\mathbb C^n$. Let $p$ be a boundary point of $\Omega$. There exists a neighborhood $U$ of $p$ so that for all $q_1,q_2\in U\cap \Omega$,
\begin{align}\label{4.6}
|K_\Omega(q_1;\bar q_2)|\lesssim |B^\#(\pi(q_1),t)|^{-1},
\end{align}
where $t=|\rho(q_1)|+|\rho(q_2)|+\inf\{\epsilon>0:q_2\in D(q_1,\epsilon)\}$. Moreover, there exists a constant $c$ such that $q_1, q_2\in B^\#(\pi(q_1),c\delta)$.
\end{thm}
Here the estimate (\ref{4.6}) follows from (\ref{3.1}). Recall that the polydisc $D(q,\delta)$ induces a global quasi-metric \cite{McNeal3} on $\Omega$. Then a triangle inequality argument using this quasi-metric yields the containment  $q_2\in B^\#(\pi(q_1),c\delta)$.

\section{Convex Body domination}
 \subsection{Convex body average $\llangle f \rrangle_{Q}$} For a 
function $f\in L^1(Q)$ on some set $Q\subseteq \mathbb C^n$ with values in $\mathbb C^d$, we recall the definition of the convex body average $\llangle f \rrangle_{Q}$ in \cite{NPTV}:
\[\llangle f \rrangle_{Q}:=\{\langle\varphi f\rangle_{Q}:\varphi \text{ is any complex valued function on } Q \text{ with }\|\varphi\|_\infty\leq 1\}.\]

By its definition, the set $\llangle f \rrangle_{Q}$ is symmetric, convex, and compact. For a collection of dyadic tents $\{\mathcal T_l\}_{l=1}^N$, we define the dyadic operator $L$ by
\[Lf=\sum_{l=1}^N\sum_{\hat K_j^k\in \mathcal T_l}\llangle f\rrangle_{\hat K_j^k}1_{\hat K_j^k}.\]
This operator takes a vector-valued function and returns a set-valued function.
\begin{lem}[{\hspace{1sp}\cite[Lemma 2.7]{NPTV}}]\label{lem2.13}
Let $f\in L^1(Q,\mathbb C^d)$ and let $g(x)\in \llangle f\rrangle_Q$ a.e. on  $Q$. Then there exists a measurable function $H:Q\times Q\to \mathbb C$ with $\|H\|_\infty\leq |Q|^{-1}$ such that 
\[g(z)=\int_QH(z,w)f(w)dV(w).\]
\end{lem}
Thus to estimate $L$, it suffices to estimate all operators of the form 
\[f\mapsto \sum_{\hat K_j^k\in \mathcal T}\int_{\hat K^k_j} H_{\hat K^k_j}(z,w)f(w)dV(w),\]
where $H_{\hat K^k_j}$  is supported on $\hat K^k_j\times \hat K^k_j$ and satisfies $\|H_{\hat K^k_j}\|_\infty\leq |{\hat K^k_j}|^{-1}$.

\begin{thm} \label{thm2.14} There exists a constant $C$ such that for any $f\in L^1(\Omega, \mathbb C^d)$ and $z\in \Omega$,
	\[P(f)(z)\in CL f(z).\]
\end{thm}

\begin{proof}Let $\{\mathcal T_l\}_{l=1}^N$ be collections of dyadic tents induced by $\{\mathcal Q_l\}_{l=1}^N$ in Lemma \ref{lem3.8}. If $z\in \Omega\backslash N_{\epsilon_0}(\mathbf b\Omega)$, then $|K_\Omega(z;\cdot)|\lesssim 1$ and 
	\[P(f)(z)=\int_\Omega K_\Omega(z;\bar w) f(w)dV(w)\in C\llangle f\rrangle_\Omega\subseteq C Lf(z),\]
	for some constant $C$.
	If  $z\in \Omega\cap N_{\epsilon_0}(\mathbf b\Omega)$, then we can set $k_0:=\max\{k\in \mathbb N:z\in B^\#(\pi(z),s^{-k}\delta)\}$.
	Theorem \ref{thm4.5} yields that for $l\in \{1,\dots, k_0-1\}$ and point $w\in B^\#(\pi(z),s^{-l}\delta)\backslash B^\#(\pi(z),s^{-l-1}\delta)$, 
	\[|K_\Omega(z;\bar w)|\lesssim |B^\#(\pi(z),s^{-l}\delta)|^{-1}.\]
	By Lemma \ref{lem3.8}, we can find $k_0$ tents $\{\hat K_j\}_{j=1}^{k_0}$ from $\{\mathcal T_l\}_{l=1}^N$  such that for each $j$, \[B^\#(\pi(z),s^{-j}\delta)\subseteq\hat K_j\;\text{ and }\;|\hat K_j|\approx |B^\#(\pi(z),s^{-j}\delta)|.\]
	Set \begin{align}U_0&=\Omega\backslash B^\#(\pi(z),\delta)\nonumber\\ U_j&=B^\#(\pi(z),s^{-j}\delta)\backslash B^\#(\pi(z),s^{-j-1}\delta) \;\text{ for }\;j=1,\dots,k_0-1.\nonumber\end{align}
	Then there is a constant $C$ independent of $j$ and $z$ such that
	\[\int_{U_j}K_\Omega(z;\bar w) f(w)dV(w)\in C\llangle f\rrangle_{\hat K_j}.\]
	This containment implies
\begin{align}
P(f)(z)&=\int_\Omega K_\Omega(z;\bar w) f(w)dV(w)\nonumber\\&=\sum_{j=0}^{k_0-1}\int_{U_j}K_\Omega(z;\bar w) f(w)dV(w)\in CLf(z),
\end{align}	
which completes the proof.
\end{proof}
\section{A step averaging weight $\mathcal W$ and its properties}
In this section, we introduce the step averaging weight which will play a crucial rule in the proof of Theorem \ref{thm1.2}.
\begin{de}\label{de4.1}Let $W$ be a matrix weight over $\Omega$. For a collection $\mathcal T$ of dyadic tents, we define a step weight $\mathcal W$ to be
\[\mathcal W(z):=\sum_{\hat K^k_j\in \mathcal T}\langle W\rangle_{ K^k_j}1_{K^k_j}(z).\]
For collections $\{\mathcal T_l\}_{l=1}^N$ of tents that satisfies both Lemmas \ref{lem3.8} and \ref{lem2.9}, we 
define $\tilde{\mathcal W}$ to be
\[\tilde{\mathcal W}(z):=\sum_{l=1}^N\mathcal W_l(z)=\sum_{l=1}^N\sum_{\hat K^k_j\in \mathcal T_l}\langle W\rangle_{ K^k_j}1_{K^k_j}(z).\]
\end{de}
Since the $K^k_j$'s are disjoint, it is easy to see that 
\[\mathcal W^{-1}(z)=\sum_{\hat K^k_j\in \mathcal T}\langle W\rangle^{-1}_{ K^k_j}1_{K^k_j}(z).\]
\begin{lem}\label{lem4.2}
	If $W$ is a matrix $\mathcal B_2$ weight, then $\mathcal W$ is also a matrix $\mathcal B_2$ weight with \begin{equation}
		\label{4.1}
		\sup_{\hat K^k_j\in \mathcal T}\bignm\langle \mathcal W^{-1}\rangle^{1/2}_{\hat K^k_j}\langle \mathcal W\rangle^{1/2}_{\hat K^k_j}\bignm^2\leq \mathcal B_2(W).\end{equation}
\end{lem}
\begin{proof}
By its definition, $\langle \mathcal W\rangle_{\hat K^k_j}=\langle  W\rangle_{\hat K^k_j}$. Given a unit vector $v\in \mathbb C^d$, 
\begin{align}\label{4.4}
\nm\langle \mathcal W^{-1}\rangle^{1/2}_{\hat K^k_j}\langle \mathcal W\rangle^{1/2}_{\hat K^k_j}v\nm^2&=\langle  \langle  \mathcal W\rangle^{1/2}_{\hat K^k_j}\langle \mathcal W^{-1}\rangle_{\hat K^k_j}\langle \mathcal W\rangle^{1/2}_{\hat K^k_j}v,v\rangle\nonumber\\&=|\hat K^k_j|^{-1}\sum_{K^l_m\subseteq \hat K^k_j}|K^l_m|\langle  \langle  W\rangle^{1/2}_{\hat K^k_j}\langle W\rangle_{K^l_m}^{-1}\langle  W\rangle^{1/2}_{\hat K^k_j}v,v\rangle\nonumber\\&=|\hat K^k_j|^{-1}\sum_{K^l_m\subseteq \hat K^k_j}|K^l_m|\nm\langle W\rangle_{K^l_m}^{-1/2}\langle  W\rangle^{1/2}_{\hat K^k_j}v\nm^2
\nonumber\\&=|\hat K^k_j|^{-1}\sum_{K^l_m\subseteq \hat K^k_j}|K^l_m|\nm\langle W\rangle_{K^l_m}^{-1/2}\langle W^{-1}\rangle_{K^l_m}^{-1/2}\langle W^{-1}\rangle_{K^l_m}^{1/2}\langle  W\rangle^{1/2}_{\hat K^k_j}v\nm^2
\nonumber\\&\leq|\hat K^k_j|^{-1}\sum_{K^l_m\subseteq \hat K^k_j}|K^l_m|\nm\langle W^{-1}\rangle_{K^l_m}^{1/2}\langle  W\rangle^{1/2}_{\hat K^k_j}v\nm^2
=\nm\langle W^{-1}\rangle_{\hat K^k_j}^{1/2}\langle  W\rangle^{1/2}_{\hat K^k_j}v\nm^2.
\end{align} The last inequality above uses the fact that $\langle  W^{-1}\rangle^{1/2}_{\hat K^k_j}\langle  W\rangle^{1/2}_{\hat K^k_j}$ is expanding. See  \cite[Corollary 3.3]{TV97}. 
\end{proof}

By the technique of \cite[Lemma 3.5]{TV97}, the weight $\langle \mathcal W^{-1}\rangle^{{1}/{2}}_{\hat K^k_j}\mathcal W(z)\langle \mathcal W^{-1}\rangle^{{1}/{2}}_{\hat K^k_j}$ is also a matrix weight {satisfying (\ref{4.1}).} Moreover, for a unit vector $v\in\mathbb C^d$, the scalar function \[\omega(v;z):=\left\langle\langle \mathcal W^{-1}\rangle^{{1}/{2}}_{\hat K^k_j}\mathcal W(z)\langle \mathcal W^{-1}\rangle^{{1}/{2}}_{\hat K^k_j}v,v\right\rangle,\]
is a scalar weight {satisfying
\[\sup_{\hat K^k_j\in \mathcal T}\langle \mathcal \omega(v;\cdot)^{-1}\rangle_{\hat K^k_j}\langle \mathcal \omega(v;\cdot)\rangle_{\hat K^k_j}\lesssim \mathcal B_2(W).\]}

Note that $\omega(v;z)$ is a fixed constant over each kube $K^k_j$. 
The next lemma shows the $\omega(v;z)$ satisfies the reverse H\"older inequality. In the case $\Omega=\mathbb D$, the lemma follows from \cite[Theorem 1.7]{Aleman}.
\begin{lem}\label{lem4.5}
	The function $\omega(v;z)$ satisfies the reverse H\"older inequality over the tent $\hat K^k_j\in \mathcal T$, i.e. there exists a constant $r>1$ such that
	\begin{equation}\label{4.5}
	\langle\omega(v;z)^r\rangle^{1/r}_{\hat K^k_j}\lesssim \langle\omega(v;z)\rangle_{\hat K^k_j}.
	\end{equation}
	Moreover, the constant $r$ above can be chosen such that $r-1\approx \mathcal B^{-1}_2(W)$.
\end{lem}
\begin{proof}
	On the fixed tent $\hat K^k_j$ and a constant $R>1$ to be determined, we apply a corona decomposition of $\hat K^k_j$ as follows:
	\begin{enumerate}
		\item First, we define the set of stopping children  $\mathcal L(\hat K^l_m)$ of a tent $\hat K^l_m\in\mathcal T$:
		\[\mathcal L(\hat K^l_m):=\left\{\text{ maximal tents }\hat K^p_q \text{ such that }\frac{\int_{\hat K^p_q}\omega(v;z)dV}{|\hat K^p_q|}>R\frac{\int_{\hat K^l_m}\omega(v;z)dV}{|\hat K^l_m|}\right\}.\]
		\item Set $\mathcal L_1:=\mathcal L(\hat K^k_j)$ and for $j\geq 2$, set $\mathcal L_i:=\cup_{\hat K^l_m\in\mathcal L_{i-1}}\mathcal L(\hat K^l_m)$ and $\mathcal L:=\cup_{i\geq 1}\mathcal L_i$. 
	\end{enumerate}
	Then we have for $\hat K^l_m\in \mathcal L_i$,
	\begin{equation}\label{4.61}\sum_{\hat K^p_q\subseteq\hat K^l_m, \hat K^p_q\in \mathcal L_{m+1}}|\hat K^p_q|\leq R^{-1}|\hat K^l_m|.\end{equation}
	Since tents in $\mathcal L_m$ are maximal, we also have for $\hat K^l_m\in \mathcal L_i$,
	\begin{equation}\label{4.7}R^i\frac{\int_{\hat K^k_j}\omega(v;z)dV}{|\hat K^k_j|}\leq \frac{\int_{\hat K^l_m}\omega(v;z)dV}{|\hat K^l_m|}\leq (cR)^i\frac{\int_{\hat K^k_j}\omega(v;z)dV}{|\hat K^k_j|}.\end{equation}
	Here $c$ can be chosen  to be the supremum of  $|\hat K^{l}_{m}|/|\hat K^{l+1}_{m^\prime}|$ over all tents $\hat K^{l}_{m}$ and $\hat K^{l+1}_{m^\prime}$ satisfying $\hat K^{l+1}_{m^\prime}\subseteq \hat K^{l}_{m}$. By Lemma \ref{3.10}, $c$ is finite. 
	Recall the maximal operator $M_{\mathcal T}$ given by
	\[M_{\mathcal T}f(z)=\sup_{\hat K^k_j\in \mathcal T}\langle |f|\rangle_{\hat K^k_j}1_{\hat K^k_j}(z).\]
	Then (\ref{4.7}) implies
	\[M_{\mathcal T}(1_{\hat K^k_j}\omega(v;z))\lesssim \left(1+\sum_i\sum_{\hat K^l_m\in \mathcal L_i}(cR)^i1_{\hat K^l_m}\right)\frac{\int_{\hat K^k_j}\omega(v;z)dV}{|\hat K^k_j|}.\]
	Since $\omega(v;z)$ is a fixed constant over each kube $K^l_m$, we have for $z\in K^l_m$, \[\omega(v;z)=\frac{\int_{K^l_m}\omega(v;z)dV}{|K^l_m|}\lesssim\frac{\int_{\hat K^l_m}\omega(v;z)dV}{|\hat K^l_m|}.\] Thus  $\omega(v;z)\lesssim M_{\mathcal T}(\omega(v;\cdot)1_{\hat K^k_j})(z)$. By applying a scalar multiplication to  $\omega(v;z)$, we may assume that $\langle\omega(v;\cdot)\rangle_{\hat K^k_j}=1$. Then, for a point $z\in \hat K^k_j\backslash\bigcup_{\hat K^{l^\prime}_{m^\prime}\in\mathcal L_i}\hat K^{l^\prime}_{m^\prime}$, we have
	\begin{align}\label{4.81}
	M_{\mathcal T}(\omega(v;\cdot)1_{\hat K^k_j})(z)\approx \langle\omega(v;\cdot)\rangle_{\hat K^l_m}\approx c^i\langle\omega(v;\cdot)\rangle_{\hat K^k_j}=c^i,
	\end{align}
	which also implies
	\begin{align}\label{4.82}
	\log(e+\omega(v;z))\lesssim i.
	\end{align}
	Since $\omega(v;z)$ is a scalar $\mathcal B_2$ weight, the proof of \cite[Lemma 7.1.9(b)]{Grafakos} implies that \[\|M_{\mathcal T}\|_{L^2(\omega^{-1}(v;\cdot))}\lesssim \mathcal B_2(\omega(v;\cdot))\leq \mathcal B_2(W).\] Then H\"older's inequality implies
\begin{align}\int_{\hat K^k_j}M_{\mathcal T}(\omega(v;\cdot)1_{\hat K^k_j})dV&=\int_{\hat K^k_j}M_{\mathcal T}(\omega(v;\cdot)1_{\hat K^k_j})\omega^{-1/2}(v;\cdot)\omega^{1/2}(v;\cdot)dV\nonumber\\&\leq \left(\int_{\hat K^k_j}M^2_{\mathcal T}(\omega(v;\cdot)1_{\hat K^k_j})\omega^{-1}(v;\cdot)dV\right)^{1/2} \left(\int_{\hat K^k_j}\omega^{}(v;\cdot)dV\right)^{1/2}\nonumber\\&\leq \mathcal B_2(\omega(v,\cdot))\left(\int_{\hat K^k_j}\omega^2(v;\cdot)\omega^{-1}(v;\cdot)dV\right)^{1/2} \left(\int_{\hat K^k_j}\omega^{}(v;\cdot)dV\right)^{1/2}\nonumber\\&\leq\mathcal B_2(W)\int_{\hat K^k_j}\omega(v;z)dV(z).\end{align}
	This together with (\ref{4.81}) and (\ref{4.82}) implies, 
	\begin{align}
	\mathcal B_2(W)\int_{\hat K^k_j}\omega(v;z)dV(z)\gtrsim&\int_{\hat K^k_j}M_{\mathcal T}(\omega(v;\cdot)1_{\hat K^k_j})dV\nonumber\\
	\gtrsim&\int_{\hat K^k_j}\sum_{i\geq 1}\sum_{\hat K^l_m\in \mathcal L_i}\langle\omega(v;\cdot)\rangle_{\hat K^l_m} 1_{\hat K^l_m}dV\nonumber\\
	=&\int_{\hat K^k_j}\omega(v;\cdot)\sum_{i\geq 1}i\sum_{\hat K^l_m\in \mathcal L_i}1_{\hat K^k_j\backslash\cup_{\hat K^{l^\prime}_{m^\prime}\in\mathcal L_i}\hat K^{l^\prime}_{m^\prime}}+1dV\nonumber\\\approx&\int_{\hat K^k_j}\omega(v;z)\log(e+M_{\mathcal T}(\omega(v;\cdot))(z))dV
	\nonumber\\\gtrsim&\int_{\hat K^k_j}\omega(v;z)\log(e+\omega(v;z))dV.
	\end{align}
	Then an argument as in the proof of  \cite[Theorem 4.1]{Duoandikoetxea} implies for $0<\beta<1$, there exists a constant $\alpha\approx  c(\beta)\exp\{-\mathcal B_2(W)\}$ such that for $\hat K^l_m\subseteq \hat K^k_j$ with \[\int_{\hat K^l_m}\omega(v;z)dV\leq \beta \int_{\hat K^k_j}\omega(v;z)dV,\]  we have \[|\hat K^l_m|\leq \alpha|\hat K^k_j|.\] 
	Now we are ready to show (\ref{4.5}). Choose  $\beta=(4c)^{-1}$ where $c$ is the constant in (\ref{4.7}). Set $R=\alpha^{-1}\approx\exp\{\mathcal B_2(W)\}$. Then we have 
	\begin{align}\label{4.11}
	\langle\omega(v;\cdot)^r\rangle_{\hat K^k_j}&\lesssim \langle (M_{\mathcal T}(\omega(v;\cdot)1_{\hat K^k_j})(z))^r\rangle_{\hat K^k_j}\nonumber\\&\lesssim \frac{\left(\int_{\hat K^k_j}\omega(v;z)dV\right)^r}{|\hat K^k_j|^{1+r}}\int_{\hat K^k_j}\left(1+\sum_i\sum_{\hat K^l_m\in \mathcal L_i}(cR)^i1_{\hat K^l_m}\right)^rdV\nonumber\\&\approx \frac{\left(\int_{\hat K^k_j}\omega(v;z)dV\right)^r}{|\hat K^k_j|^{1+r}}\int_{\hat K^k_j}\left(1+\sum_{i\geq 1}\sum_{\hat K^l_m\in \mathcal L_i}(cR)^{ir}1_{\hat K^l_m}\right)dV\nonumber\\&= \left(\frac{\int_{\hat K^k_j}\omega(v;z)dV}{|\hat K^k_j|}\right)^r+\frac{\left(\int_{\hat K^k_j}\omega(v;z)dV\right)^r}{|\hat K^k_j|^{1+r}}\sum_{i\geq 1}\sum_{\hat K^l_m\in \mathcal L_i}(cR)^{ir}|\hat K^l_m|\nonumber\\&\leq \left(\frac{\int_{\hat K^k_j}\omega(v;z)dV}{|\hat K^k_j|}\right)^r+\frac{\left(\int_{\hat K^k_j}\omega(v;z)dV\right)^{r-1}}{|\hat K^k_j|^{r}}\sum_{i\geq 1}c^{ir}R^{i(r-1)}\sum_{\hat K^l_m\in \mathcal L_i}\int_{\hat K^l_m}\omega(v;z)dV
	\nonumber\\&\leq \left({\langle \omega(v;)\rangle_{\hat K^k_j}}\right)^r(1+\sum_{i\geq 1}c^{ir}R^{i(r-1)}(4c)^{-i})={\langle \omega(v;)\rangle^r_{\hat K^k_j}}(1+\sum_{i\geq 1}(cR)^{i(r-1)}\frac{1}{4^{i}}).
	\end{align}
	By choosing $r$ such that $(cR)^{(r-1)}=2$, we obtain (\ref{4.5}). Applying the natural logarithm it yields:
\[(r-1)\ln (cR)=\ln 2.\]
Since $R\approx \exp\{\mathcal B_2(W)\}$, we obtain 
$r-1\approx \mathcal B^{-1}_2(W)$.
\end{proof}

\begin{lem}\label{lem4.4}
	For any holomorphic vector-valued function $f$ on $\Omega$, \[\|f\|_{L^2(W)}\lesssim\|f\|_{L^2(\tilde{\mathcal W})}.\]
\end{lem}
\begin{proof}Recall $K^{\prime-1}$ and $D(z,cs^{-k})$ from Lemma \ref{lem2.9}. For each $z\in \Omega$, we define a $z$-dependent subset $D^\star(z)$ as follows:
	For $z\in K^{\prime-1}$, we set $D^\star(z)$ to be the maximal Euclidean ball in $K^{-1}$  centered at $z$. For $z\in \Omega \backslash K^{\prime-1}$, we set $D^\star(z)$ to be the polydisc $D(z,cs^{-k})$  from Lemma \ref{lem2.9}. Then for all $z\in \Omega$
	\[\langle\tilde{\mathcal W}(z)f(z), f(z)\rangle\gtrsim\langle\langle W\rangle_{D^\star(z)}f(z), f(z)\rangle.\]
We claim that there is a constant $c<1$ such that for all $\zeta\in cD^\star(z)$, there exists a polydisc (or ball) $D^\prime(\zeta)$ centered $\zeta$ satisfying the following properties: (1) $z\in D^\prime(\zeta)$;
(2) $D^\prime(\zeta)\subseteq D^\star(z)$; (3) $|D^\prime(\zeta)|\approx|D^\star(z)|$. When $D^\star(z)$ is a Euclidean ball, this is obvious. When $D^\star(z)$  is a polydisc, the claim follows from (\ref{2.11}) and (\ref{2.12}). Then Fubini's theorem implies that
	\begin{align}
	\int_{\Omega}\langle\tilde{\mathcal  W} f(z),f(z)\rangle dV(z)&\gtrsim \int_\Omega|D^\star(z)|^{-1}\int_{D^\star(z)}\langle W(\zeta) f(z),f(z)\rangle dV(\zeta)dV(z)\nonumber\\&\gtrsim \int_\Omega|D^\prime(\zeta)|^{-1}\int_{D^\prime(\zeta)}\langle W(\zeta) f(z),f(z)\rangle dV(z)dV(\zeta).
	\end{align}
	Note that $\langle W(\zeta) f(z),f(z)\rangle$ is subharmonic in $z$, we have 
	\[\int_\Omega |D^\prime(\zeta)|^{-1} \int_{D^\prime(\zeta)}\langle W(\zeta) f(z),f(z)\rangle dV(z)dV(\zeta)\geq \int_\Omega\langle W(\zeta) f(\zeta),f(\zeta)\rangle dV(\zeta). \]
	This implies $\|f\|_{L^2(W)}\lesssim\|f\|_{L^2(\tilde{\mathcal W})}$.
\end{proof}
As a consequence of Lemma \ref{lem4.4}, the norm of $P$ on the weighted space $L^2(W)$ can be dominated by the norm of $P$ on the weighted space $L^2(\tilde{\mathcal W})$.
\begin{thm}\label{thm4.3}Let $W$ be  a matrix $\mathcal B_2$ weight on $\Omega$. Then $\|P\|_{L^2(W)}\lesssim \mathcal B^{1/2}_2(W)\|P\|_{L^2(\tilde{\mathcal W})}$.
\end{thm}
\begin{proof}
	By Lemma \ref{lem4.4} and a duality argument, we have 
	\[\|P\|_{L^2(W)}\lesssim \|P:L^2(W)\to L^2(\tilde{\mathcal W})\|=\|P:L^2(\tilde{\mathcal W}^{-1})\to L^2(W^{-1})\|.\]
	Set $\tilde{\mathcal W}_{-1}(z):=\sum_{l=1}^{N}\sum_{\hat K^k_j\in \mathcal T_l}\langle W^{-1}\rangle_{ K^k_j}1_{K^k_j}(z).$
 Applying Lemma \ref{lem4.4} again to the weight $W^{-1}$ yields
	\[\|P\|_{L^2(W)}\lesssim\|P:L^2(\tilde{\mathcal W}^{-1})\to L^2(W^{-1})\|\lesssim \|P:L^2(\tilde{\mathcal W}^{-1})\to L^2(\tilde{\mathcal W}_{-1})\|.\]
	For any $z\in \Omega$, we consider those kubes $K^l_m$ of $\hat K^k_j\in \bigcup_{l=1}^N\mathcal T_l$  that contain the point $z$. If $K^l_m\ni z$, then $\text{dist}(K^l_m,\mathbf b \Omega)\approx \text{dist}(z,\mathbf b \Omega)$. By Lemma \ref{lem3.6} and Definition \ref{de3.9}, there  exists  a constant $r$ such that $\pi(K^l_m)\subseteq B(\pi(z),r)$ and $|K^l_m|\approx |B^\#(\pi(z),r)|$. Then Lemma \ref{lem3.8} implies that there is a dyadic tent $\hat K^k_j\in \bigcup_{l=1}^N\mathcal T_l$ such that $\hat K^k_j\supseteq B^\#(\pi(z),r)$ with $|\hat K^k_j|\approx |B^\#(\pi(z),r)|$. Thus
	\begin{align}\label{4.111}
	\tilde{\mathcal W}(z)&=\sum_{\hat K^l_m\in \bigcup_{l=1}^N\mathcal T_l}\langle W\rangle_{K^l_m}1_{K^l_m}(z)\lesssim N\langle W\rangle_{\hat K^k_j},\\\label{4.12}
	\tilde{\mathcal W}_{-1}(z)&=\sum_{\hat K^l_m\in \bigcup_{l=1}^N\mathcal T_l}\langle W^{-1}\rangle_{K^l_m}1_{K^l_m}(z)\lesssim N\langle W^{-1}\rangle_{\hat K^k_j}. 
	\end{align}
	Inequalities (\ref{4.111}) and (\ref{4.12}) yield
	\[\nm\tilde{\mathcal W}^{1/2}_{-1}(z)\tilde{\mathcal W}^{1/2}(z)\nm^2\lesssim \nm\langle W^{-1}\rangle_{\hat K^k_j}^{1/2}\langle W\rangle_{\hat K^k_j}^{1/2}\nm^2\leq \mathcal B_2(W).\] Hence, we have 
	$ \nm \tilde{\mathcal W}^{1/2}_{-1}(z)f(z)\nm \leq \mathcal B^{1/2}_2(W)\nm\tilde{\mathcal W}^{-1/2}(z)f(z)\nm$, which implies
	\[ \|P\|_{L^2(W)}\lesssim\|P:L^2(\tilde{\mathcal W}^{-1})\to L^2(\tilde{\mathcal W}_{-1})\|\lesssim \mathcal B^{1/2}_2(W)\|P\|_{L^2(\tilde{\mathcal W}^{-1})}=\mathcal B^{1/2}_2(W)\|P\|_{L^2(\tilde{\mathcal W})}.\]
\end{proof}

\section{Proof of Theorem \ref{thm1.2}}
By Theorem \ref{thm4.3} and the fact that $\tilde{\mathcal W}=\sum_{l=1}^{N}\mathcal W_l$, it suffices to show that \[\|P\|_{L^2({\mathcal W})}\lesssim \mathcal B^{3/2}_2(W),\] where  $\mathcal W$ is a weight defined as in Definition \ref{de4.1}.

We begin with the following lemma which can be viewed as a special case of the Carleson Embedding Theorem for $\mathcal T$.  We provide the proof below for the completeness of the paper.
\begin{lem}\label{lem2.16} Let $\mathcal T$ be a collection of dyadic tents as in Lemma \ref{3.10}.
Then for any scalar-valued measurable function $f$ on $\Omega$ and for any $p\in (1,\infty)$,
	\[\sum_{\hat K^k_j\in \mathcal T}\langle f\rangle^p_{\hat K^k_j} |\hat K^k_j|\lesssim (p^\prime)^p\|f\|^p_{L^p(\Omega)}.\]
\end{lem}

\begin{proof}
Recall the maximal operator $M_{\mathcal T}$ given by
\[M_{\mathcal T}f(z)=\sup_{\hat K^k_j\in \mathcal T}\langle |f|\rangle_{\hat K^k_j}1_{\hat K^k_j}(z).\]
For $p\in (1,\infty)$, a standard argument yields that $\|M_{\mathcal T}\|_{L^p(\Omega)}\leq p^\prime$.
Recall also that for each tent  $\hat K^k_j$, its induced kube $K^k_j$ satisfying $|\hat K^k_j|\approx |K^k_j|$. Then
\[\sum_{\hat K^k_j\in\mathcal T }\langle |f|\rangle^p_{\hat K^k_j} |\hat K^k_j|\lesssim\sum_{\hat K^k_j\in\mathcal T }\langle |f|\rangle^p_{\hat K^k_j} |K^k_j|\leq \int_{\Omega}(M_{\mathcal T}f)^pdV\leq (p^\prime)^p\|f\|^p_{L^p(\Omega)}.\]
\end{proof}

By Lemma \ref{lem2.13} and Theorem \ref{thm2.14}, we need to estimate the norm of the operator
\[\sum_{\hat K^k_j\in \mathcal T}\int_{\hat K_j} H_{\hat K_j}(z,w)f(w)dV(w)1_{\hat K_j}(z),\]
on the weighted space $L^2(\mathcal W)$. Here $H_{\hat K_j}$ satisfying $|H_{\hat K_j}(z,w)|\leq|\hat K^k_j|^{-1}$. This is equivalent to estimating  the operator 
\[Q_W(f)(z)=\sum_{\hat K^k_j\in \mathcal T}\int_{\hat K_j} H_{\hat K_j}(z,w)\mathcal W^{1/2}(z)\mathcal W^{-1/2}(w)f(w)dV(w)1_{\hat K_j}(z),\]
over the unweighted space  $L^2(\Omega,\mathbb C^d)$. Set \[\mathcal W_{-1}(z)=\sum_{\hat K^k_j\in \mathcal T}\langle W^{-1}\rangle_{ K^k_j}1_{K^k_j}(z).\]
For  vector-valued functions $f,g\in L^2(\Omega, \mathbb C^d)$,
\begin{align}\label{(4.1)}
&|\langle Q_W(f), g\rangle|\nonumber\\\leq &\sum_{\hat K^k_j\in \mathcal T}\left|\int_{\Omega}\left\langle\int_{\hat K^k_j} H_{\hat K_j}(z,w)\mathcal W^{-1/2}(w)f(w)dV(w)1_{\hat K^k_j}(z), \mathcal W^{1/2}(z)g(z)\right\rangle dV(z)\right|\nonumber\\=&\sum_{\hat K^k_j\in \mathcal T}\left|\int_{\Omega}\int_{\hat K^k_j} \left\langle H_{\hat K^k_j}(z,w)\langle \mathcal W^{-1}\rangle^{1/2}_{\hat K^k_j} \langle \mathcal W^{-1}\rangle^{-1/2}_{\hat K^k_j} \mathcal W^{-1/2}(w)f(w), \mathcal W^{1/2}(z)g(z)1_{\hat K^k_j}(z)\right\rangle dV(w)dV(z)\right|\nonumber\\=&\sum_{\hat K^k_j\in \mathcal T}\left|\int_{\hat K^k_j}\int_{\hat K^k_j} \left\langle\langle \mathcal W^{-1}\rangle^{-1/2}_{\hat K^k_j} \mathcal W^{-1/2}(w)f(w), \langle \mathcal W^{-1}\rangle^{1/2}_{\hat K^k_j} \mathcal W^{1/2}(z)g(z) H_{\hat K^k_j}(z,w)\right\rangle dV(w)dV(z)\right|\nonumber\\=&\sum_{\hat K^k_j\in \mathcal T}\left|\int_{\hat K^k_j}\left\langle\langle \mathcal W^{-1}\rangle^{-1/2}_{\hat K^k_j} \mathcal W^{-1/2}(w)f(w), \int_{\hat K^k_j} \langle \mathcal W^{-1}\rangle^{1/2}_{\hat K^k_j}\mathcal W^{1/2}(z)g(z) H_{\hat K^k_j}(z,w)dV(z)\right\rangle dV(w)\right|\nonumber\\\leq&\sum_{\hat K^k_j\in \mathcal T}\int_{\hat K^k_j}\nm\langle \mathcal W^{-1}\rangle^{-1/2}_{\hat K^k_j} \mathcal W^{-1/2}(w)f(w)\nm\bignm \int_{\hat K^k_j} \langle \mathcal W^{-1}\rangle^{1/2}_{\hat K^k_j}\mathcal W^{1/2}(z)g(z) H_{\hat K^k_j}(z,w)dV(z)\bignm dV(w)\nonumber\\\leq&\sum_{\hat K^k_j\in \mathcal T}\int_{\hat K^k_j}\nm\langle \mathcal W^{-1}\rangle^{-1/2}_{\hat K^k_j} \mathcal W^{-1/2}(w)f(w)\nm \int_{\hat K^k_j} \bignm\langle \mathcal W^{-1}\rangle^{1/2}_{\hat K^k_j}\mathcal W^{1/2}(z)g(z) \bignm|\hat K^k_j|^{-1}dV(z) dV(w)\nonumber\\\lesssim&\sum_{\hat K^k_j\in \mathcal T}\left\langle \nm\langle \mathcal W^{-1}\rangle^{-1/2}_{\hat K^k_j}\mathcal W^{-1/2}f\nm\right\rangle_{\hat K^k_j}\left\langle \nm\langle \mathcal W^{-1}\rangle^{1/2}_{\hat K^k_j}\mathcal W^{1/2}g\nm\right\rangle_{\hat K^k_j}|K^k_j|\nonumber\\\leq&\left(\sum_{\hat K^k_j\in \mathcal T}\left\langle \nm\langle \mathcal W^{-1}\rangle^{-1/2}_{\hat K^k_j}\mathcal W^{-1/2}f\nm\right\rangle^2_{\hat K^k_j}|K^k_j|\right)^{\frac{1}{2}}\left(\sum_{\hat K^k_j\in \mathcal T}\left\langle \nm\langle \mathcal W^{-1}\rangle^{1/2}_{\hat K^k_j}\mathcal W^{1/2}g\nm\right\rangle^2_{\hat K^k_j}|K^k_j|\right)^{\frac{1}{2}}.
\end{align}
Set operators $S_{1,W}$ and $S_{2,W}$ to be as follows:
\begin{align}
S_{1,W}(f)(z)&:=\sum_{\hat K^k_j\in \mathcal T}\left\langle \nm\langle \mathcal W^{-1}\rangle^{-1/2}_{\hat K^k_j}\mathcal W^{-1/2}f\nm\right\rangle^2_{\hat K^k_j}1_{K^k_j}(z),\nonumber\\S_{2,W}(g)(z)&:=\sum_{\hat K^k_j\in \mathcal T}\left\langle \nm\langle \mathcal W^{-1}\rangle^{1/2}_{\hat K^k_j}\mathcal W^{1/2}g\nm\right\rangle^2_{\hat K^k_j}1_{K^k_j}(z).\nonumber
\end{align}
For $S_{2,W}(g)$, H\"older's inequality implies
\begin{equation}\label{4.3}\langle \nm\langle \mathcal W^{-{1}/{2}}\rangle^{{1}/{2}}_{\hat K^k_j} \mathcal W^{{1}/{2}}g\nm\rangle_{\hat K_j}\leq \langle \nm\langle \mathcal W^{-{1}/{2}}\rangle^{{1}/{2}}_{\hat K^k_j} \mathcal W^{{1}/{2}}\nm^{2r}\rangle^{{1}/{2r}}_{\hat K_j}\langle\nm g\nm^{(2r)^\prime}\rangle^{{1}/{(2r)^\prime}}_{\hat K_j}.\end{equation}
Choosing $r$ as in Lemma \ref{lem4.5} and applying the lemma,  we have
\begin{align}
\langle\nm\langle \mathcal W^{-1}\rangle^\frac{1}{2}_{\hat K^k_j}\mathcal W^{\frac{1}{2}}\nm^{2r}\rangle_{\hat K^k_j}&=|\hat K^k_j|^{-1}\int_{\hat K^k_j}\nm\langle \mathcal W^{-1}\rangle^\frac{1}{2}_{\hat K^k_j}\mathcal W(w)\langle \mathcal W^{-1}\rangle^\frac{1}{2}_{\hat K^k_j}\nm^{r}dV(w)\nonumber\\&\lesssim|\hat K^k_j|^{-1}\int_{\hat K^k_j}\left(\text{Tr}\left(\langle \mathcal W^{-1}\rangle^\frac{1}{2}_{\hat K^k_j}\mathcal W(w)\langle \mathcal W^{-1}\rangle^\frac{1}{2}_{\hat K^k_j}\right)\right)^{r}dV(w)
\nonumber\\&\lesssim|\hat K^k_j|^{-1}\int_{\hat K^k_j}\max_{1\leq k\leq d}\left\langle\langle \mathcal W^{-1}\rangle^\frac{1}{2}_{\hat K^k_j}\mathcal W(w)\langle \mathcal W^{-1}\rangle^\frac{1}{2}_{\hat K^k_j}e_k,e_k\right\rangle_{\mathbb C^d}^{r}dV(w)
\nonumber\\&\leq\sum_{k=1}^d|\hat K^k_j|^{-1}\int_{\hat K^k_j}\left\langle\langle \mathcal W^{-1}\rangle^\frac{1}{2}_{\hat K^k_j}\mathcal W(w)\langle \mathcal W^{-1}\rangle^\frac{1}{2}_{\hat K^k_j}e_k,e_k\right\rangle_{\mathbb C^d}^{r}dV(w)
\nonumber\\&\lesssim\sum_{k=1}^d\left(|\hat K^k_j|^{-1}\int_{\hat K^k_j}\left\langle\langle \mathcal W^{-1}\rangle^\frac{1}{2}_{\hat K^k_j}\mathcal W(w)\langle \mathcal W^{-1}\rangle^\frac{1}{2}_{\hat K^k_j}e_k,e_k\right\rangle_{\mathbb C^d}dV(w)\right)^{r}\nonumber\\&\lesssim  (\mathcal B_2(W))^{r}.
\end{align}
Applying this inequality to (\ref{4.3}) gives
\[S_{2,W}(g)(z)\lesssim (\mathcal B_2(W))^{\frac{1}{2}} \sum_{\hat K^k_j\in \mathcal T}\langle\nm g\nm^{(2r)^\prime}\rangle_{\hat K^k_j}^{\frac{1}{(2r)^\prime}}1_{ K^k_j}(z),\]
Set $h=\|g\|^{(2r)^\prime}$ and $p=2/(2r)^\prime$.
Then 
\[\|S_{2,W}(g)\|^2_{L^2(\Omega)}\lesssim \mathcal B_2(W) \sum_{\hat K^k_j\in \mathcal T}\langle h\rangle_{\hat K^k_j}^{p} |\hat K^k_j|\lesssim \mathcal B_2(W)(p^\prime)^p\|h\|^p_{L^p(\Omega)}=\mathcal B_2(W)(p^\prime)^p\|g\|^2_{L^2(\Omega)},\]
where the second inequality above follows from Lemma \ref{lem2.16}.
Since $p^\prime=2+\frac{1}{\epsilon}\lesssim \frac{1}{\epsilon}$ and $p\leq 1+\epsilon$, we have $(p^\prime)^p\lesssim 1/\epsilon=(r-1)^{-1}$. From the choice of $r$ in Lemma \ref{lem4.5}, \[(r-1)^{-1}\approx \mathcal B_2(W).\]
Therefore, \[\|S_{2,W}g(z)\|^2_{L^2(\Omega)}\lesssim\mathcal B^2_2(W)\|g\|^2_{L^2(\Omega)}.\]

We turn to estimate the norm of $S_{1,W}f$. Note that
\[\int_{\hat K^k_j}\langle \mathcal W^{-1}\rangle^{-\frac{1}{2}}_{\hat K^k_j}\mathcal W^{-1}(w)\langle \mathcal W^{-1}\rangle^{-\frac{1}{2}}_{\hat K^k_j}dV(w)=\mathbf I_d.\]
Using this fact and going through the same argument as $S_{2,W}g$ yield that
\[\|S_{1,W}f(z)\|^2_{L^2(\Omega)}\lesssim\mathcal B_2(W)\|f\|^2_{L^2(\Omega)}.\]
Combining these inequalities, we obtain
\[\langle Q_wf,g\rangle\lesssim (\mathcal B_2(W))^{\frac{3}{2}}\|f\|_{L^2(\Omega)}\|g\|_{L^2(\Omega)},\]
which shows the desired estimate  $\|P\|_{L^2(\Omega,\mathcal W)}\leq \| Q_W\|_{L^2(\Omega)}\lesssim  (\mathcal B_2(W))^{\frac{3}{2}}$. 
\section{Remarks}
\paragraph{1}  By the proof of Theorem \ref{thm2.14}, it's not hard to see that the positive Bergman operator $P^+$ also belongs to a convex body valued sparse operator. The same argument in Section 5 then yields the following corollary for the weighted norm of  $P^+$ on the  weighted space  $L^2(\tilde{\mathcal W})$:
\begin{co} Let $\Omega$ be a simple domain.
Let $W$ be a matrix $\mathcal B_2$ weight. Let $\tilde{\mathcal W}$ be a weight constructed based on $W$
as in Definition \ref{de4.1}. Then $\|P^+\|_{L^2(\tilde{\mathcal W})}\lesssim \mathcal B_2^{3/2}(W)$.\end{co} 
We are unable to obtain estimates for $\|P^+\|_{L^2(W)}$ since Lemma \ref{lem4.4} is not available for a general vector-valued function $f\in L^2(W)$.
\vskip 5pt
\paragraph{2} As pointed out in \cite{HWW2} for the scalar-valued case, the products of averages of $W$ and $W^{-1}$ over the whole domain and over the small tents will have different impacts on the estimate for the weighted norm of the projection $P$. However, our estimate in Theorem \ref{thm1.2} is less likely to be sharp. Hence we didn't consider this difference in here for the simplicity of the argument.

\bibliographystyle{alpha}
\bibliography{2} 

\begin{thebibliography}{HWW20b}

\bibitem[AC12]{AC}
A.~Aleman and O.~Constantin.
\newblock The {B}ergman projection on vector-valued {$L^2$}-spaces with
  operator-valued weights.
\newblock {\em J. Funct. Anal.}, 262(5):2359--2378, 2012.

\bibitem[APR19]{Aleman}
A.~Aleman, S.~Pott, and M.~C. Reguera.
\newblock Characterizations of a limiting class ${B}^\infty$ of
  {B}{\'e}koll{\'e}-{B}onami weights.
\newblock {\em Rev. Mat. Iberoam.}, 35(6):1677--1692, 2019.

\bibitem[BB78]{BB78}
D.~Bekoll\'{e} and A.~Bonami.
\newblock In\'{e}galit\'{e}s \`a poids pour le noyau de {B}ergman.
\newblock {\em C. R. Acad. Sci. Paris S\'{e}r. A-B}, 286(18):A775--A778, 1978.

\bibitem[BB00]{BaloghBonk}
Z.~M. Balogh and M.~Bonk.
\newblock {G}romov hyperbolicity and the {K}obayashi metric on strictly
  pseudoconvex domains.
\newblock {\em Comment. Math. Helv.}, 75(3):504--533, 2000.

\bibitem[Bek82]{Bekolle}
D.~Bekoll\'{e}.
\newblock In\'{e}galit\'{e} \`a poids pour le projecteur de {B}ergman dans la
  boule unit\'{e} de {${\bf C}^{n}$}.
\newblock {\em Studia Math.}, 71(3):305--323, 1981/82.

\bibitem[BPW16]{BPW}
K.~Bickel, S.~Petermichl, and B.~D. Wick.
\newblock Bounds for the {H}ilbert transform with matrix {$A_2$} weights.
\newblock {\em J. Funct. Anal.}, 270(5):1719--1743, 2016.

\bibitem[BW16]{BW2016}
K.~Bickel and B.~D. Wick.
\newblock A study of the matrix {C}arleson embedding theorem with applications
  to sparse operators.
\newblock {\em J. Math. Anal. Appl.}, 435(1):229--243, 2016.

\bibitem[CG01]{CG01}
M.~Christ and M.~Goldberg.
\newblock Vector {$A_2$} weights and a {H}ardy-{L}ittlewood maximal function.
\newblock {\em Trans. Amer. Math. Soc.}, 353(5):1995--2002, 2001.

\bibitem[D'A82]{D'Angelo82}
J.~P. D'Angelo.
\newblock Real hypersurfaces, orders of contact, and applications.
\newblock {\em Ann. of Math. (2)}, 115(3):615--637, 1982.

\bibitem[DMRO16]{Duoandikoetxea}
J.~Duoandikoetxea, F.~J. Mart\'{\i}n-Reyes, and S.~Ombrosi.
\newblock On the {$A_\infty$} conditions for general bases.
\newblock {\em Math. Z.}, 282(3-4):955--972, 2016.

\bibitem[GHK20]{Hu}
C.~Gan, B.~Hu, and I.~Khan.
\newblock Dyadic decomposition of convex domains of finite type and
  applications.
\newblock 2020.
\newblock (preprint) \url{https://arxiv.org/abs/2001.07200}.

\bibitem[Gol03]{Goldberg03}
M.~Goldberg.
\newblock Matrix {$A_p$} weights via maximal functions.
\newblock {\em Pacific J. Math.}, 211(2):201--220, 2003.

\bibitem[Gra14]{Grafakos}
L.~Grafakos.
\newblock {\em Classical {F}ourier analysis}, volume 249 of {\em Graduate Texts
  in Mathematics}.
\newblock Springer, New York, third edition, 2014.

\bibitem[HK12]{HK}
T.~Hyt\"onen and A.~Kairema.
\newblock Systems of dyadic cubes in a doubling metric space.
\newblock {\em Colloq. Math.}, 126:1--33, 2012.

\bibitem[HMW73]{HMW}
R.~Hunt, B.~Muckenhoupt, and R.~Wheeden.
\newblock Weighted norm inequalities for the conjugate function and {H}ilbert
  transform.
\newblock {\em Trans. Amer. Math. Soc.}, 176:227--251, 1973.

\bibitem[HPV19]{HPV}
T.~Hyt\"{o}nen, S.~Petermichl, and A.~Volberg.
\newblock The sharp square function estimate with matrix weight.
\newblock {\em Discrete Anal.}, pages Paper No. 2, 8, 2019.

\bibitem[HW20]{ZhenghuiWick2}
Z.~Huo and B.~D. Wick.
\newblock Weighted estimates for the {B}ergman projection on the {H}artogs
  triangle.
\newblock {\em J. Funct. Anal.}, 2020.
\newblock (accepted) \url{https://arxiv.org/abs/1904.10501}.

\bibitem[HWW20a]{HWW}
Z.~Huo, N.~A. Wagner, and B.~D. Wick.
\newblock A {B}ekoll\'e-{B}onami class of weights for certain pseudoconvex
  domains.
\newblock 2020.
\newblock (preprint) \url{https://arxiv.org/abs/2001.08302}.

\bibitem[HWW20b]{HWW2}
Z.~Huo, N.~A. Wagner, and B.~D. Wick.
\newblock {B}ekoll\'e-{B}onami estimates on some pseudoconvex domains.
\newblock 2020.
\newblock (preprint) \url{https://arxiv.org/abs/2001.07868}.

\bibitem[Hyt12]{Hytonen}
T.~P. Hyt\"{o}nen.
\newblock The sharp weighted bound for general {C}alder\'{o}n-{Z}ygmund
  operators.
\newblock {\em Ann. of Math. (2)}, 175(3):1473--1506, 2012.

\bibitem[IKP17]{IKP}
J.~Isralowitz, H-K. Kwon, and S.~Pott.
\newblock Matrix weighted norm inequalities for commutators and paraproducts
  with matrix symbols.
\newblock {\em J. Lond. Math. Soc. (2)}, 96(1):243--270, 2017.

\bibitem[Lac17]{Lacey2017}
M.~T. Lacey.
\newblock An elementary proof of the ${A}_2$ bound.
\newblock {\em Isr. J. Math.}, 217(1):181--195, Mar 2017.

\bibitem[Ler13]{Lerner13}
A.~K. Lerner.
\newblock A simple proof of the {$A_2$} conjecture.
\newblock {\em Int. Math. Res. Not. IMRN}, (14):3159--3170, 2013.

\bibitem[LP21]{LP21}
A.~Limani and S.~Pott.
\newblock {Sparse Lerner operators and domination in infinite dimensions}.
\newblock 2021.
\newblock (preprint) \url{https://arxiv.org/pdf/2103.17005}.

\bibitem[McN91]{McNeal91}
J.~D. McNeal.
\newblock Local geometry of decoupled pseudoconvex domains.
\newblock In Klas Diederich, editor, {\em Complex Analysis}, pages 223--230,
  Wiesbaden, 1991. Vieweg+Teubner Verlag.

\bibitem[McN94a]{McNeal3}
J.~D. McNeal.
\newblock The {B}ergman projection as a singular integral operator.
\newblock {\em J. Geom. Anal.}, 4(1):91--103, 1994.

\bibitem[McN94b]{McNeal2}
J.~D. McNeal.
\newblock Estimates on the {Bergman} kernels of convex domains.
\newblock {\em Adv. Math.}, 109(1):108--139, 1994.

\bibitem[McN03]{McNeal2003}
J.~D. McNeal.
\newblock Subelliptic estimates and scaling in the
  {$\overline\partial$}-{N}eumann problem.
\newblock In {\em Explorations in complex and {R}iemannian geometry}, volume
  332 of {\em Contemp. Math.}, pages 197--217. Amer. Math. Soc., Providence,
  RI, 2003.

\bibitem[NPTV17]{NPTV}
F~Nazarov, S.~Petermichl, S.~Treil, and A.~Volberg.
\newblock Convex body domination and weighted estimates with matrix weights.
\newblock {\em Adv. Math.}, 318:279--306, 2017.

\bibitem[NT96]{NT96}
F.~L. Nazarov and S.~R. Treil.
\newblock The hunt for a {B}ellman function: applications to estimates for
  singular integral operators and to other classical problems of harmonic
  analysis.
\newblock {\em Algebra i Analiz}, 8(5):32--162, 1996.

\bibitem[Pet07]{Petermichl1}
S.~Petermichl.
\newblock The sharp bound for the {H}ilbert transform on weighted {L}ebesgue
  spaces in terms of the classical {$A_p$} characteristic.
\newblock {\em Amer. J. Math.}, 129(5):1355--1375, 2007.

\bibitem[PR13]{Pott}
S.~Pott and M.~C. Reguera.
\newblock Sharp {B}\'ekoll\'e estimates for the {B}ergman projection.
\newblock {\em J. Funct. Anal.}, 265(12):3233--3244, 2013.

\bibitem[RTW17]{Rahm}
R.~Rahm, E.~Tchoundja, and B.~D. Wick.
\newblock Weighted estimates for the {B}erezin transform and {B}ergman
  projection on the unit ball.
\newblock {\em Math. Z.}, 286(3-4):1465--1478, 2017.

\bibitem[TV97a]{TV1997}
S.~Treil and A.~Volberg.
\newblock Continuous frame decomposition and a vector
  {H}unt-{M}uckenhoupt-{W}heeden theorem.
\newblock {\em Ark. Mat.}, 35(2):363--386, 1997.

\bibitem[TV97b]{TV97}
S.~Treil and A.~Volberg.
\newblock Wavelets and the angle between past and future.
\newblock {\em J. Funct. Anal.}, 143(2):269--308, 1997.

\bibitem[Vol97]{Volberg97}
A.~Volberg.
\newblock Matrix {$A_p$} weights via {$S$}-functions.
\newblock {\em J. Amer. Math. Soc.}, 10(2):445--466, 1997.

\bibitem[Wit00]{Wittwer}
J.~Wittwer.
\newblock A sharp estimate on the norm of the martingale transform.
\newblock {\em Math. Res. Lett.}, 7(1):1--12, 2000.

\end{thebibliography}
\end{document}